\documentclass[letterpaper, 10 pt, conference]{ieeeconf}  

\IEEEoverridecommandlockouts                              
\usepackage{epsfig} 
\usepackage{amsmath,amssymb,amsfonts,mathrsfs}
\usepackage{color}
\usepackage{balance}
\usepackage{booktabs,tabularx}
\usepackage{multirow}
\usepackage{mathtools}
\usepackage{pdfrender}
\usepackage{todonotes}
\usepackage{trfsigns}

\usepackage{centernot}
\usepackage{color}
\usepackage[acronym]{glossaries}
\usepackage[ruled]{algorithm2e}
\usepackage{graphicx,subcaption}
\usepackage{tabu}
\usepackage{cite}
\usepackage[official]{eurosym}
\usepackage{xcolor,calc}

\graphicspath{{figures/}}

\newcommand{\R}{\mathbb{R}}

\newcommand{\N}{\mathbb{N}}

\newcommand{\mc}[1]{\mathcal{#1}}

\newcommand{\bx}{\bs{x}}

\newcommand{\bs}[1]{\boldsymbol{#1}}
\newcommand{\bsone}{\boldsymbol{1}}

\newtheorem{theorem}{Theorem}

\newtheorem{definition}{Definition}

\newtheorem{example}{Example}
\newtheorem{remark}{Remark}
\newtheorem{assumption}{Assumption}

\makeglossaries
\newacronym{GNEP}{GNEP}{generalized Nash equilibrium problem}
\newacronym{NEP}{NEP}{ Nash equilibrium problem}
\newacronym{MI-NEP}{MI-NEP}{mixed-integer Nash equilibrium problem}
\newacronym{MI}{MI}{mixed-integer}
\newacronym{VI}{VI}{variational inequality}
\newacronym{ICRF}{ICRF}{integer-compatible regularization function}
\newacronym{BR}{BR}{best-response}
\newacronym{KKT}{KKT}{Karush--Kuhn--Tucker}
\newacronym{w.r.t.}{w.r.t.}{with respect to}
\newacronym{LF-NE}{LF-NE}{leader-follower Nash equilibrium}
\newacronym{MPEC}{MPEC}{mathematical program with equilibrium constraint}
\newacronym{MPECs}{MPECs}{mathematical program with equilibrium constraints}

\newglossaryentry{GNE}
{
	name={GNE},
	description={generalized Nash equilibrium},
	first={\glsentrydesc{GNE} (\glsentrytext{GNE})},
	plural={generalized Nash equilibrium},
	descriptionplural={generalized Nash equilibria},
	firstplural={\glsentrydescplural{GNE} (GNE)}
}

\newglossaryentry{MI-NE}
{
	name={MI-NE},
	description={mixed-integer Nash equilibrium},
	first={\glsentrydesc{MI-NE} (\glsentrytext{MI-NE})},
	plural={mixed-integer Nash equilibrium},
	descriptionplural={mixed-integer Nash equilibria},
	firstplural={\glsentrydescplural{MI-NE} (MI-NE)}
}

\newglossaryentry{MI-GNE}
{
	name={MI-GNE},
	description={mixed-integer generalized Nash equilibrium},
	first={\glsentrydesc{MI-GNE} (\glsentrytext{MI-GNE})},
	plural={mixed-integer generalized Nash equilibria},
	descriptionplural={mixed-integer generalized Nash equilibria},
	firstplural={\glsentrydescplural{MI-GNE} (MI-GNE)}
}

\newcommand{\st}{\textrm{s.t.}}

\newcommand{\set}[1]{\{#1\}}

\newcommand{\defset}[3][\defsep]{\set{#2#1#3}}
\newcommand{\Defset}[3][\defsep]{\Set{#2#1#3}}

\newcommand{\define}{\mathrel{{\mathop:}{=}}}

\newcommand{\scrX}{\mathcal{X}}
\newcommand{\scrZ}{\mathcal{Z}}
\newcommand{\NE}{\mathbf{NE}}

\title{\LARGE \bf
  A Gauss-Seidel method for solving multi-leader-multi-follower games
}

\author{Barbara Franci$^1$, Filippo Fabiani$^2$, Martin Schmidt$^3$, and Mathias Staudigl$^4$
  \thanks{$^1$ Department of Advanced Computing Sciences, Maastricht University, NL--6200 MD, Maastricht, The Netherlands {\tt \footnotesize (b.franci@maastrichtuniversity.nl)}.}
  \thanks{$^2$ IMT School for Advanced Studies Lucca, Piazza San Francesco 19, 55100 Lucca, Italy ({\tt \footnotesize filippo.fabiani@imtlucca.it}).}
  \thanks{$^3$ Department of Mathematics at Trier University, Universitätsring 15,
    54296 Trier, Germany {\tt\footnotesize(martin.schmidt@uni-trier.de)}. }
  \thanks{$^4$ University of Mannheim, Department of Mathematics, B6, 68159 Mannheim,
    Germany {\tt\footnotesize(mathias.staudigl@uni-mannheim.de)}. }
}

\begin{document}

\maketitle
\thispagestyle{empty}
\pagestyle{empty}

\begin{abstract}
We design a computational approach to find equilibria in a class of
Nash games possessing a hierarchical structure. By using tools from
mixed-integer optimization and the characterization of variational
equilibria in terms of the Karush--Kuhn--Tucker conditions, we propose
a mixed-integer game formulation for solving this challenging
class of problems. Besides providing an equivalent reformulation, we design a
proximal Gauss--Seidel method with global convergence guarantees in
case the game enjoys a potential structure. We finally corroborate the
numerical performance of the algorithm on a novel instance of the
ride-hail market problem.


\end{abstract}

\section{Introduction}
\label{sec:introduction}

In this work, we are interested in solving a class of hierarchical
games involving a set of $N$ leaders who are strategically attached to
$M_{i}$ followers, where $i\in[N]=\{1,\ldots,N\}$. The leaders
are in turn involved in a Nash game among the other leading
agents. Mathematically, we model each leader $i$ as a rational
decision-maker with loss function $F^{i}(x^{i},\bs x^{-i},\bs y^{i})$,
depending on a private decision~$x^{i}$, a tuple~$\bs x^{-i}$
representing the opponent leaders' decisions, and a
tuple~$\bs{y}^{i}$, which constitutes an~$M_{i}$-tuple
$\bs y^{i}\coloneqq(y_{1}^{i},\ldots,y^{i}_{M_{i}})$ describing the lower-level
agents' decisions.
We are thus interested in computing an equilibrium in a game in which
leader~$i$, subject to local constraints $\mathcal{X}_i$, solves
\begin{equation}
  \label{eq:Main}
  \min_{x^{i},\bs y^{i}} \quad F^{i}(x^{i},\bs x^{-i},\bs y^{i})
  \text{ s.t. }
  x^{i} \in \scrX^{i}, \, \bs{y}^{i}\in\NE^{i}(x^{i}),
\end{equation}
where the followers' decision strategy $\bs y^{i}$ constitutes an equilibrium
of a \gls{GNEP}, parametric in the leader's decision $x^i$, i.e., $\bs
y^{i} \in \NE^{i}(x^{i})$.

Such models have attracted great interest in control and operations
research \cite{sherali1984multiple,hintermuller2014several}, recently
also from a stochastic
perspective~\cite{cui2023computation,cui2023regularized}, and belong to the family of \emph{multi-leader-multi-follower games}.
Under some assumptions, this class of problems
reduces to mathematical programs with equilibrium constraints (MPECs)
\cite{luo1996mathematical} when there is only one leader who controls
the equilibrium correspondence of a set of
followers \cite{hintermuller2011smooth,hintermuller2014several,de2023bilevel,Aussel:2020aa,Pang:2005aa}. By adopting the terminology of this literature, we refer to
problem~\eqref{eq:Main} as the \emph{optimistic formulation} of the
multi-leader-multi-follower game \cite{fabiani2021local,kulkarni2015existence}.

Although multi-leader-multi-follower games are a natural model for
several applications in engineering systems, to date existence
results for equilibria are restricted to special classes of games
\cite{Pang:2005aa}. A commonly made assumption is that for
every upper-level decision $x^{i}$, the set of equilibria $\NE^i(x^{i})$ is a
singleton. This is quite restrictive, especially in the presence of a
GNEP at the lower level.
Moreover, a challenging problem is the lack
of analytically tractable characterizations of the equilibrium
correspondence $\NE^{i}(x^{i})$. For polynomial data, it possesses a
semi-algebraic structure, being defined as a set of polynomial
equations and inequalities \cite{BlumeZame94}, which is however not
enough to develop tractable numerical methods for computing an
equilibrium solution. In fact, the Nash correspondence will exhibit
nonconvexities, inherited from the complementarity conditions
characterizing first-order optimal points. To fill this gap, our main
contributions can be summarized as follows:
\begin{itemize}
\item We develop a global, optimization-based reformulation of the equilibrium
conditions for the proposed multi-leader-multi-follower game. Such a
reformulation brings the hierarchical game into a large-scale,
nonlinear, mixed-integer optimization problem;
\item To ensure that the problem is
numerically tractable, we recast the leader's game as a potential
game. While potential games are a relatively small class of games,
they are a classical choice in mixed-integer games
\cite{Fabiani_et_al:2022,sagratella2017algorithms,sagratella2019generalized}. Building
upon such an optimization reformulation, we propose an
iterative method based on a Gauss--Seidel iteration inspired by our
previous work on a standard (i.e., single-level) Nash equilibrium
problem~\cite{Fabiani_et_al:2022};
\item We propose a novel formulation for the ride-hail market
  problem. Inspired by
  \cite{fabiani2022stochastic,fabiani2023personalized}, the model
  described here and the equilibrium it provides are meant to help the
  platforms (i.e., the leaders) in estimating the profit and thereby
  supporting its pricing decision process, anticipating the decision
  of the drivers (i.e., the followers).
\end{itemize}




\section{Problem setting}
\label{sec:problem-statement}

To describe our hierarchical game setting, we now introduce the
different levels of strategic interactions separately.
To this end, we make use of the notation introduced in
\cite{Aussel_et_al:2023}.

\subsection{The leaders' problem}

First, we formalize the optimization problem that each leader~$i \in [N]$
aims to solve:
\begin{equation}
  \label{eq:leader-agent-1}
  \begin{aligned}
   &\underset{x^i, \bs y^{i}}{\textrm{min}} && F^i(x^i, \bs x^{-i}, \bs y^i)\\
    &~\st && G^i(x^i) \geq 0,~\bs y^i  \in\NE^{i}(x^{i}),
  \end{aligned}
\end{equation}
%
%
%
where $x^{i} \in \R^{n^i}$
%
%
is the vector of the $i$-th leader's decisions, while $\bs x^{-i} =
(x_j)_{j \in [N]\setminus\{i\}} \in \R^{n-n^i}$ collects the decisions
of the $i$-th leader's opponents and
$n\define\sum_{i\in[N]}^{}n^{i}$.
The decisions of the followers of leader $i$ are stacked in
$\bs y^i = (y^i_\nu)_{\nu \in [M_i]} \in \R^{p^i}$, with
$y^i_\nu \in\R^{p^i_\nu}$, so that $p^{i} \define \sum_{\nu \in [M_i]}
p^i_\nu$.

The cost function of leader $i\in[N]$ is $F^i :\R^{n}\times \R^{p^i}
\to \R$ while $G^{i} : \R^{n^i} \to \R^{d^i}$ defines $d_i$ private
constraints.
With an eye to \eqref{eq:Main}, we thus define $\scrX^{i} \define
\defset{x^{i}\in\R^{n^{i}}}{G^{i}(x^{i})\geq 0 }$. Furthermore,
$\NE^i(x^i)$ denotes the set of generalized Nash equilibria (GNE) of
the non-cooperative game among the $M_i$ followers attached to the
$i$-th leader, discussed later in \S \ref{subsec:followers}.

Each problem in~\eqref{eq:leader-agent-1}, which satisfies the
conditions stated next, defines an MPEC in which $\bs y^i$ is not
strictly within the control of the leader, but corresponds to an
optimistic conjecture \cite{kulkarni2015existence,fabiani2021local}.
We hence assume the following.
\begin{assumption}\label{ass:1}
  For all $i\in[N]$, it holds that:
  \begin{enumerate}
  \item[(i)] $F^{i}:\R^{n}\times\R^{p^{i}}\to\R$ is jointly continuous in both
    arguments.
  \item[(ii)] $G^i:\R^{n^{i}}\to\R^{d^{i}}$ is continuous and concave
    in each coordinate, i.e.,
    $G^{i}(x^{i})=[g^{i}_{1}(x^{i});\ldots;g^{i}_{d^{i}}(x^{i})]$ is a
    column vector and each coordinate function $g^{i}_{k}(\cdot)$ is a
    concave function.
    \hfill$\square$
  \end{enumerate}
\end{assumption}

\smallskip

Before describing the GNEP involving each leader's followers, we
introduce the solution notion employed for the considered class of
multi-leader-multi-follower games.
\begin{definition}\label{def:LFNE}
  A tuple $\{(\bar{x}^{i,\ast},\bar{\bs y}^{i,\ast})\}_{i\in[N]}$ is a
  \gls{LF-NE} if, for all $i \in [N]$, it solves the collection of
  optimistic MPECs in \eqref{eq:leader-agent-1}.
\hfill$\square$
\end{definition}

\subsection{The generalized Nash game played by the followers}
\label{subsec:followers}

The optimization problem of each leader is challenging due to the
presence of the equilibrium constraints, aggregating the decisions of
all lower-level agents indexed by $\nu \in [M_{i}]$. The followers
attached to leader $i$ are in turn involved in a strategic competition
represented by a generalized Nash game at the lower
level. Specifically, we consider a situation in which the lower-level
problem of leader $i$ hosts $M_{i}$ interdependent minimization
problems of the form:
\begin{subequations}
  \label{eq:ll-game}
  \begin{align}
    &\underset{y^i_\nu}{\textrm{min}} && f^i_\nu(x^i, y^i_\nu, \bs y^i_{-\nu})
    \\
    &~\st && D^i_\nu y^i_\nu \geq e^i_{\nu} - A^{i}_{\nu} x^i - \sum_{\mu \in [M_i]\setminus\{\nu\}}
    D^i_{\nu,\mu} y^i_\mu,
    \label{eq:ll-game:local}\\
    &&& B^i x^i + \sum_{\mu \in [M_i]}^{} E^i_\mu y^i_\mu \geq c^i,
    \label{eq:ll-game:coupling}
    \end{align}
    \end{subequations}
for $\nu=1,\ldots,M_{i}$. These minimization problems give rise
to a GNEP with the following features:
i) Both the objectives $f^{i}_{\nu}$ and the constraints depend on the
upper-level decision variable $x^{i}$ chosen by the leader; ii) The
control variables of the followers affect the cost criterion
$f^{i}_{\nu}$, as well as the constraints.
Therefore the $M_i$ followers are involved in a \gls{GNEP}, parametric
in the leader's decision variable $x^{i}$.
The restrictions of follower $i$'s decision variable
$y^{i}$ feature private constraints \eqref{eq:ll-game:local}, as well
as the shared constraints \eqref{eq:ll-game:coupling} that both depend on the
decisions of the other lower-level agents. The cost function of
follower $\nu\in[M_{i}]$ is a real-valued function
$f^{i}_{\nu}:\R^{n^{i}}\times\R^{p^{i}}\to\R$.
In addition, $D^i_{\nu,\mu} \in \R^{r^i_\nu \times p^i_\mu}$,
$A^i_{\nu} \in \R^{r^i_{\nu} \times n^i}, D^{i}_{\nu}\coloneqq
D^{i}_{\nu,\nu}$, and $e^{i}_{\nu} \in \R^{r^i_{\nu}}$ define the
private constraints of follower $\nu$, while the matrix $B^{i}\in\R^{d^{i}\times n^{i}}$
measures the impact of the upper-level decisions on the lower-level
shared constraints. Finally, $E^i_\nu \in \R^{d^{i} \times p^{i}_{\nu}}$, and
$c^{i} \in \R^{d^{i}}$ define the coupling constraints of the
followers. Let $\Gamma^{i}(x^{i})$ indicate the lower-level game associated with leader $i$, parameterized by the leader decision $x^{i}$. We denote the feasible set of the $\nu$-th
lower-level agent by $\Omega^i_\nu(x^i, \bs y^i_{-\nu})$, which
amounts to a closed, polyhedral set.
We define the $\nu$-th agent's \gls{BR} correspondence
as the solution set to the parametric problem:
\begin{equation*}
  \beta^{i}_{\nu}(x^{i}, \bs y^{i}_{-\nu})
  \coloneqq \underset{y^{i}_{\nu}\in\Omega^{i}_{\nu}(x^{i}, \bs
    y^{i}_{-\nu})}{\textrm{argmin}}~f^{i}_{\nu}(x^{i},y^{i}_{\nu},\bs
  y^{i}_{-\nu}).
\end{equation*}

\begin{definition}
  A tuple $\bar{\bs y}^{i} = (\bar{y}^{i}_{\nu})_{\nu \in
    [M_i]}\in\R^{p^{i}}$ is a GNE of the game
  $\Gamma^{i}(x^{i})$ if
  $\bar{y}^{i}_{\nu}\in\beta^{i}_{\nu}(x^{i},\bar{\bs y}^{i}_{-\nu})$
  for all $\nu\in[M_{i}]$.
  The set of GNE for $\Gamma^{i}(x^{i})$ is
  denoted by $\NE^{i}(x^{i})$.
  \hfill$\square$
\end{definition}

In particular, we are interested in the ``well-posed case'' in which
$\NE^{i}(x^{i})\neq\emptyset$ for all $x^{i}\in\scrX^{i}$, a condition
guaranteed beforehand via standard assumptions on
the data of \eqref{eq:ll-game} \cite{Facchinei_2010_GNEP}.
\begin{assumption}\label{ass:well-posed}
  For all $x^{i}\in\scrX^{i}$, $\NE^{i}(x^{i})\neq\emptyset$.
  \hfill$\square$
\end{assumption}

Moreover, later in the paper we will rely on the \gls{KKT}
characterization of the set $\NE^{i}(x^{i})$, and therefore we will
require the following assumption.
\begin{assumption}\label{ass:Slater}
  For all $\nu\in[M_{i}]$, given $x^{i}$ and $\bs y^{i}_{-\nu}$,
 $y^{i}_{\nu}\mapsto f^{i}_{\nu}(x^{i},y^{i}_{\nu},\bs y^{i}_{-\nu})$
  is continuously differentiable and convex.
  \hfill$\square$
\end{assumption}

In addition, our numerical scheme will focus on the computation of a
specific solution of the followers' game, known as
\emph{variational equilibrium}, widely employed in game theory \cite{Facchinei_2010_GNEP,facchinei2011decomposition}.
To introduce this solution concept, let
$\Omega^{i}(x^{i})=\left\{\bs y^{i} \in \R^{p_i} \mid
  y^{i}_{\nu}\in\Omega^{i}_{\nu}(x^{i},\bs y^{i}_{-\nu}) \, \forall \nu\in[M_{i}]\right\}$
be the feasible set of the followers and consider the
\emph{pseudogradient mapping}
$$
V^{i}(x^{i},\bs y^{i})\coloneqq (\nabla_{y^i_\nu} f^i_\nu(x^{i},y^{i}_\nu,\bs y^{i}_{-\nu}))_{\nu\in[M_i]}.
$$
\begin{definition}
    Given some $x^{i}\in\scrX^{i}$, inducing a GNEP
    $\Gamma^{i}(x^{i})$, an $M_{i}$-tuple $\bar{\bs y}^{i,\ast}$ is a
    variational equilibrium if $(\bs y^{i}-\bar{\bs y}^{i,\ast})^\top
    V^{i}(x^{i},\bar{\bs y}^{i,\ast})\ge0$ for all $\bs
    y^{i}\in\Omega^{i}(x^{i})$.
    \hfill$\square$
\end{definition}

\smallskip

\subsection{Special instances}

The class of lower-level \glspl{GNEP} we consider is quite general, as
illustrated by the following special instances.

\begin{example}[Hierarchical games]
If $D^{i}_{\nu,\mu}$ and $B^{i}$ as well as all the $E^{i}_{\mu}$'s
are all zero matrices and $c^{i}=0$, then each follower $\nu$ only
faces the private coupling constraints $D^{i}_{\nu}y^{i}_{\nu}\geq
e^{i}_{\nu}-A^{i}_{\nu}x^{i}$.
The lower-level problem attached to leader $i$ then becomes a standard
Nash equilibrium problem. Accordingly, the restriction $\bs
y^{i}\in\NE^i(x^{i})$ could then be reformulated as the hierarchical
game problem with the following data. Define the pseudo-gradient
$V^{i}(x^{i},\bs
y^{i})=(\nabla_{y^{i}_{\nu}}f^{i}_{\nu}(x^{i},y^{i}_{\nu},\bs
y^{i}_{-\nu}))_{\nu\in[M_{i}]}$ as well as the sets
$\mathcal{Y}^{i}_{\nu}(x^{i})\coloneqq\{y^{i}_{\nu}\in\R^{p^{i}_{\nu}}\mid
D^{i}_{\nu}y^{i}_{\nu}\geq e^{i}_{\nu}-A^{i}_{\nu}x^{i}\}$, and
$\mathcal{Y}^{i}(x^{i})\coloneqq\prod_{\nu\in[M_{i}]}\mathcal{Y}^{i}_{\nu}(x^{i})$. Then,
our leader-follower game can be compactly reformulated as the
mathematical problem with variational inequality (VI) constraints:
\begin{equation*}
\begin{aligned}
  &\underset{x^{i},\bs y^{i}}{\textrm{min}}&&F^{i}(x^{i},\bs x^{-i},\bs y^{i})\\
  &\text{ s.t.} && x^{i}\in\scrX^{i},\\
  &&&\langle{ V^{i}(x^{i},\bs y^{i}),\tilde{\bs y}^{i}-\bs
    y^{i}\rangle}\geq 0 \; \forall \tilde{\bs
    y}^{i}\in\mathcal{Y}^{i}(x^{i}).
\end{aligned}
\end{equation*}
Under Assumption \ref{ass:Slater}, identifying $\NE^{i}(x^{i})$ with
the solution to a VI is a classic concept for Nash games
\cite{Facchinei_2010_GNEP}, a class of games that is attracting
several research interest
\cite{cui2023computation,cui2023regularized}.
   \hfill$\square$
\end{example}

\begin{example}[Bilevel structure]
  If $M_{i}=1$ for all $i$, then each leader defines parameters for
  a single follower, who then solves the parametric optimization problem
  $$
  \min_{y^{i}} \quad f^{i}_{1}(x^{i},y^{i})\quad\text{s.t.} \quad
  B^{i}x^{i}+E^{i}_{1}y^{i}\geq c^{i}, \ D^{i}_{1}y^{i}\geq e^{i}.
  $$
  This is a special case when $A^{i}_{1}=0$. Under Assumption
  \ref{ass:Slater}, the first-order optimality conditions for the
  single follower lead to the VI
  $
  \langle{\nabla_{y}f^{i}_{1}(x^{i},\bar{y}^{i}),y^{i}-\bar{y}^{i}\rangle}\geq 0
  $
  for all $y^{i}\in\mathcal{Y}^{i}(x^{i})$, where
  $\mathcal{Y}^{i}(x^{i})$ represents the feasible set of the follower's
  optimization problem. The model we aim at solving is then a
  multi-agent version of a bilevel optimization problem, reading
  explicitly (dropping the superfluous subscripts) as:
  $$
  \begin{aligned}
    &\underset{x^{i},y^{i}}{\textrm{min}} && F^{i}(x^{i},\bx^{-i},y^{i})\\
    &\text{ s.t. } &&y^{i}\in\textrm{argmin}\{f^{i}_{1}(x^{i},\tilde{y}^{i})\mid \tilde{y}^{i}\in\mathcal{Y}^{i}(x^{i})\}.
  \end{aligned}
  $$
  Clearly, if $N=1$, this model corresponds to a standard bilevel
  optimization problem \cite{Dempe-et-al:2015,Dempe_Zemkoho:2020}.
  \hfill$\square$
\end{example}


\section{An equivalent reformulation}
\label{sec:reformulation}

To obtain an equivalent reformulation of the
single-leader-multi-follower game in \eqref{eq:leader-agent-1},
parametric in the $i$-th leader's opponents $\bs x^{-i}$, we use
the characterization of a GNE by starting from the
\gls{KKT} conditions associated with follower~$\nu$. In fact, under
Assumption~\ref{ass:Slater}, any point
$y^{i}_{\nu}\in\beta^{i}_{\nu}(x^{i},\bs y^{i}_{-\nu})$ satisfies the
following set of necessary optimality conditions
\begin{equation*}
  \begin{aligned}
    &\nabla_{y^i_\nu} f^i_\nu(x^{i},y^i_\nu, \bs y^i_{-\nu}) -
    (D^i_\nu)^\top \lambda^i_\nu - (E^i_\nu)^\top \delta^i_\nu
    = 0,
    \\
    &0\leq \lambda^{i}_{\nu}\perp D^i_\nu y^i_\nu - e^i + D^i x^i +
    \sum_{\mu \in [M_i]\setminus\{\nu\}} D^i_{\nu,\mu} y^i_\mu
    \geq 0,\\
    &0 \leq \delta^{i}_{\nu} \perp B^i x^i + \sum_{\mu \in [M_i]}^{}
    E^i_\mu y^i_\mu - c^{i}
    \geq 0,
  \end{aligned}
\end{equation*}
where the variables $(\lambda^{i}_{\nu},\delta^{i}_{\nu}) \in
\R^{r^{i}_{\nu}}\times\R^{p^{i}_{\nu}}$ are the Lagrange
multipliers of agent~$\nu$. Then, in view of
\cite[Thm.~4.6]{Facchinei_2010_GNEP}, when the tuple
$(\bar{y}^{i}_{\nu},\bar{\lambda}^{i}_{\nu},\bar{\delta}^{i}_{\nu})$
satisfies the \gls{KKT} conditions for every $\nu\in[M_{i}]$, we have
that $\bar{\bs y}^{i}\in\NE^{i}(x^{i})$. Specializing this
characterization to variational equilibria, we only have to replace
the agent-specific Lagrange multipliers $\delta^{i}_{\nu}$ with a
single copy~$\delta^{i}$ of these multipliers, which solely
affects the last complementarity condition in the system above.

By replacing the \gls{GNEP} in the lower-level with the concatenation
of all \gls{KKT} conditions characterizing the followers' equilibrium,
the single-level reformulation of the $i$-th leader's
problem reads as follows:
\begin{subequations}
  \label{eq:single-level-reform}
  \begin{align}
    & \underset{x^i, \bs y^i, \lambda^i, \delta^i}{\textrm{min}} \quad
     F^i(x^i, \bs x^{-i}, \bs y^i)
    \label{eq:single-level-reform:objective} \\
    &~~~~\st \quad
     G^i(x^i)\geq 0,
    \label{eq:single-level-reform:ul-constr}\\
    &~~~~~~~~~~\nabla_{y^i_\nu} f^i_\nu(y^i_\nu, x^i, \bs y^i_{-\nu}) -
      (D^i_\nu)^\top \lambda^i_\nu - (E^i_\nu)^\top \delta^i = 0,
      \notag\\
    &\hspace{5cm}\text{ for all }\nu \in [M_i],
    \label{eq:single-level-reform:ll-constr-1}\\
    &~~~~~~~~~~D^i_\nu y^i_\nu - e^i_{\nu} + A^i_{\nu} x^i + \sum_{\mu \in
      [M_i]\setminus\{\nu\}} D^i_{\nu,\mu}
    y^i_\mu \geq 0,\notag\\
    &\hspace{5cm}\text{ for all }\nu \in [M_i], \label{eq:single-level-reform:ll-constr-2}\\
    &~~~~~~~~~~B^i x^i + \sum_{\mu \in [M_i]} E^i_\mu y^i_\mu - c^i \geq 0,
    \label{eq:single-level-reform:ll-constr-3}\\
    &~~~~~~~~~~\delta^i \geq 0, \lambda^i_\nu \geq 0, \text{ for all }\nu \in [M_i],
    \label{eq:single-level-reform:ll-constr-4}\\
    &~~~~~~~~~~ \lambda^i_\nu \perp \left(D^i_\nu y^i_\nu - e^i + A^i_{\nu} x^{i} +
      \sum_{\mu \in [M_i]\setminus\{\nu\}} D^{i}_{\nu,\mu} y^i_\mu\right)\notag\\
    &\hspace{5cm}\text{ for all }\nu \in [M_i],
    \label{eq:single-level-reform:compl-1}\\
    &~~~~~~~~~~\delta^i \perp \left( B^i x^i - \sum_{\mu \in [M_i]} E^i_\mu
      y^i_\mu - c^i \right).\label{eq:single-level-reform:compl-2}
  \end{align}
\end{subequations}

As a result, solving the reformulation in
\eqref{eq:single-level-reform}, parametric in $\bs x^{-i}$, leads to
an optimistic solution of the single-leader-multi-follower problem,
according to Definition~\ref{def:LFNE}.

\begin{theorem}
  \label{thm:correctness-reform}
  Let $(x^{i,*}, \bs y^{i,*})$ be a global optimal solution
  of the $i$-th leader's problem.
  Then, there exists $\lambda^{i,*}$ and
  $\delta^{i,*}$ so that $(x^{i,*}, \bs y^{i,*}, \lambda^{i,*},
  \delta^{i,*})$ is a global optimal solution of the single-level
  reformulation~\eqref{eq:single-level-reform}. Conversely, if
  $(x^{i,*}, \bs y^{i,*}, \lambda^{i,*},
  \delta^{i,*})$ solves the single-level
  problem~\eqref{eq:single-level-reform} globally, then $(x^{i,*},
  \bs y^{i,*})$ is a global optimal solution of the
  single-leader-multi-follower game~\eqref{eq:leader-agent-1}.
\end{theorem}
\begin{proof}
  See \cite[Thm.~3.3.8]{Aussel:2020aa}.
  \end{proof}
\begin{remark}
  Consider the following generalizations:
  \begin{enumerate}
  \item \emph{Extension to nonlinear constraints}: Under appropriate
    constraint qualifications, our framework can include
    generic convex feasible sets.
    From \cite{Aussel-et-al:2019,Dempe-Dutta:2012}, however, one
    has to analyze the respective Lagrange multipliers to
    establish the analogue of
    Theorem~\ref{thm:correctness-reform}.
  \item \emph{Multilinear terms for constraints and costs}: The constraints
    in \eqref{eq:ll-game} keep their linear nature even if one
    allows for multilinear terms.
    The associated single-level structure
    \eqref{eq:single-level-reform}, however, would inherit such
    nonlinearities and the resulting nonconvex terms.
    Analogous considerations apply to the lower-level cost functions,
    where one looses one degree of the multilinear
    polynomial by taking the gradient in the \gls{KKT}
    conditions.
    \hfill$\square$
  \end{enumerate}
\end{remark}

Next, we replace the complementarity constraints in
\eqref{eq:single-level-reform:ll-constr-2}--%
\eqref{eq:single-level-reform:compl-2}
using binary variables and a standard big-$M$ reformulation.
This yields the following \gls{MI} program:
\begin{subequations}
  \label{eq:single-level-reform-w-bins}
  \begin{alignat}{3}
    &\underset{x^i, \bs y^i, \lambda^i, \delta^i, s^i, t^i}{\textrm{min}} && F^i(x^i, \bs x^{-i}, \bs y^i)
    \label{eq:single-level-reform:objective-copy} \\
    & \hspace{.8cm}\st && G^i(x^i)\geq 0,
    \label{eq:single-level-reform:ul-constr-copy}\\
    &&& \nabla_{y^i_\nu} f^i_\nu(y^i_\nu, x^i, \bs y^i_{-\nu}) -
    (D^i_\nu)^\top \lambda^i_\nu - (E^i_\nu)^\top \delta^i_\nu = 0, \notag\\
    &&& \hspace{3cm} \text{ for all }\nu \in [M_i],
    \label{eq:single-level-reform:ll-constr-1a}\\
    &&& D^i_\nu y^i_\nu - e^i + A^i_{\nu} x^i + \sum_{\mu \in [M_i] \setminus \{\nu\};} D^i_{\nu,\mu}
    y^i_\mu \geq 0 , \notag\\
    &&& \hspace{3cm} \text{ for all }\nu \in [M_i],
    \label{eq:single-level-reform:ll-constr-2a}\\
    &&& D^i_\nu y^i_\nu - e^i + A^i_{\nu} x^i + \sum_{\mu \in [M_i] \setminus \{\nu\}} D^i_{\nu,\mu}
    y^i_\mu \notag\\
    &&& \hspace{.9cm} \leq (\bsone - s^i_{\nu}) \mathcal{M}^i_\nu \text{ for all }\nu \in [M_i],
    \label{eq:single-level-reform:ll-constr-2b}\\
    &&& B^i x^i + \sum_{\mu \in [M_i]} E^i_\mu y^i_\mu - c^i \geq 0,
    \label{eq:single-level-reform:ll-constr-3a}\\
    &&& B^i x^i + \sum_{\mu \in [M_i]} E^i_\mu y^i_\mu - c^i \leq
    (\bsone - t^i_\nu) \mathcal{N}^i_\nu,
    \label{eq:single-level-reform:ll-constr-3b}\\
    &&& \lambda^i_\nu, \delta^i_\nu \geq 0,~ \nu \in [M_i],
    \label{eq:single-level-reform:ll-constr-4a}\\
    &&& \lambda^i_\nu \leq s^i_\nu \mathcal{M}^i_\nu,~\delta^i_\nu \leq t^i_\nu \mathcal{N}^i_\nu,~\nu \in [M_i],
    \label{eq:single-level-reform:ll-constr-4b}\\
    &&& s^i_\nu \in \set{0,1},~\nu \in [M_i],
    \label{eq:single-level-reform:compl-1-binary}\\
    &&& t^i_\nu \in \set{0,1},~\nu \in [M_i],
    \label{eq:single-level-reform:compl-2-binary}
  \end{alignat}
\end{subequations}
where $\mathcal{M}^i_\nu$ and $\mathcal{N}^i_\nu$ are vectors with
sufficiently large entries.

With \eqref{eq:single-level-reform-w-bins} at hand, we can
rewrite~\eqref{eq:leader-agent-1} as a Nash game
with MI variables $z^{i}=(x^i,
\bs y^i, \lambda^i, \delta^i, s^i, t^i) \in
\R^{n^{i}}\times\R^{p^{i}}\times\R^{m^{i}}\times\R^{p^{i}}\times\R^{r^{i}}\times
\R^{M_{i}d^{i}}\times \set{0,1}^{M_{i}}\times \set{0,1}^{M_{i}}$,
satisfying \eqref{eq:single-level-reform:ul-constr-copy}--%
\eqref{eq:single-level-reform:compl-2-binary}. We emphasize that the
vector $z^{i}$ contains the leader's decision, but also all control
and dual variables of the attached followers. Hence, the leader solves
the problem of the followers by identifying a KKT point in this
optimistic formulation of the problem.
After a suitable manipulation of the variables and
constraints involved, problem~\eqref{eq:single-level-reform-w-bins} reads
\begin{equation}
  \label{eq:agent-i-final-reform}
  \forall i \in [N]: \quad
  \min_{z^i} \quad J^i(z^i, \bs z^{-i})
  \quad \st \quad z^i \in \mathcal{Z}^i,
\end{equation}
where
$$
  \scrZ^{i}
  \define
  \Defset{z^i = (x^i, \bs y^i, \lambda^i, \delta^i, s^i, t^i)}{z^i
    \text{ satisfies }
    \eqref{eq:single-level-reform:ul-constr-copy}\text{--}%
    \eqref{eq:single-level-reform:compl-2-binary}}.
$$
The proposed reformulation allows us to embed the solution of the
multi-leader-multi-follower problem as a MI-NE between the leaders~$i
\in [N]$ that solve the collection of problems in
\eqref{eq:agent-i-final-reform}. We therefore obtain a Nash game
reformulation of the original multi-leader-multi-follower game,
featuring MI restrictions at the leader level.

Next, we will show that for a relevant class of games, we can ensure
existence of a solution to \eqref{eq:agent-i-final-reform}, and
develop a numerical scheme that solves our original
multi-leader-multi-follower by relying on the Nash game
reformulation.


\section{A numerically tractable class of mixed-integer games}

Despite multi-leader-multi-follower games being a natural model for
many engineering applications, existence results for the equilibria
are restricted to special cases \cite{Pang:2005aa}.
Instead of imposing uniqueness of the lower-level
equilibrium, we exploit the \gls{KKT} conditions derived in the
previous section together with a special assumption on the agents'
cost structure. Following \cite{kulkarni2015existence}, we will
consider \emph{separable} hierarchical games with \emph{potential
  structure} over the leaders' interactions.
\begin{assumption}\label{ass:sep_pot}
  The following conditions hold true:
  \begin{enumerate}
  \item[(i)] For all $i\in[N]$, $F^{i}$ is separable, i.e., $F^{i}(x^{i},\bs x^{-i},\bs y^{i})
    = g^{i}(x^{i}, \bs x^{-i})+h^{i}(x^{i}, \bs y^{i})$;
  \item[(ii)] There exists $W:\R^{n}\to\R$ satisfying
    $g^{i}(\tilde{x}^{i}, \bs x^{-i})-g^{i}(x^{i}, \bs
    x^{-i})=W(\tilde{x}^{i}, \bs x^{-i})-W(x^{i}, \bs x^{-i})$ for all
    $i$.\hfill$\square$
  \end{enumerate}
\end{assumption}

\smallskip

Consider then the function $\pi:\R^{n}\times\R^{p}\to \R$, defined as
\begin{equation*}
\pi(\bs x, \bs y)=W( \bs x)+\sum_{i \in [N]}^{}h^{i}(x^{i}, \bs y^{i}),
\end{equation*}
as well as its lifted version to the space $\scrZ$ as
$P(\bs z)=\pi(\bs x,\bs y)$, $\bs z \in \scrZ$.
The relation in Assumption~\ref{ass:sep_pot}\,(ii) allows us to verify
that $P(\cdot)$ actually amounts to an exact potential for the
multi-leader-multi-follower game $\{(J^{i},\scrZ^{i})_{i\in[N]}\}$:
\begin{align*}
  & P(\tilde{z}^{i}, \bs z^{-i})-P(z^{i}, \bs z^{-i})
  \\
  & = W(\tilde{x}^{i}, \bs x^{-i})-W(x^{i}, \bs x^{-i})+h^{i}(\tilde{x}^{i},\tilde{ \bs y}^{i})-h^{i}(x^{i}, \bs y^{i})
  \\
  & = g^{i}(\tilde{x}^{i},\bs x^{-i})-g^{i}(x^{i}, \bs x^{-i})+h^{i}(\tilde{x}^{i},\tilde{ \bs y}^{i})-h^{i}(x^{i}, \bs y^{i})
  \\
  & = F^{i}(\tilde{x}^{i},\tilde{\bs x}^{-i},\tilde{\bs y}^{i})-F^{i}(x^{i}, \bs x^{-i}, \bs y^{i})
  \\
  & = J^{i}(\tilde{z}^{i}, \bs z^{-i})-J^{i}(z^{i}, \bs z^{-i}).
\end{align*}

We can therefore approach the \gls{LF-NE} computation as the solution to the \gls{MI} nonlinear optimization problem:
\begin{equation}\label{eq:Master}
  \underset{\bs z\in\scrZ}{\textrm{min}}~P(\bs z).
\end{equation}
\begin{assumption}\label{ass:compactlevels}
    There exists $\bs{z}(0)\in\scrZ$ for which the level set $\mathcal{L}_{P}(\bs{z}(0)):=\{\bs{z} \mid P(\bs{z})\leq P(\bs{z}(0))\}$ is compact.
    \hfill$\square$
\end{assumption}

In view of Assumptions~\ref{ass:well-posed}--\ref{ass:compactlevels}
and Theorem~\ref{thm:correctness-reform}, the master problem
\eqref{eq:Master} has a global solution and it induces a
LF-NE in the spirit of Definition~\ref{def:LFNE}. In addition, note
that Assumption~\ref{ass:compactlevels} is not overly
restrictive. Each block of $\bs{z}$ is a tuple of $z^{i}$'s consisting
of decision variables $(x^{i},\bs{y}^{i})$, Lagrange multipliers
$(\lambda^{i},\delta^{i})$ and binary variables
$(\delta^{i},s^{i},t^{i})$. Since the Lagrange multipliers are
confined to a compact domain in view of our big-$M$ reformulation,
Assumption \ref{ass:compactlevels} can hence be guaranteed by assuming
that $\pi$ has compact level sets, which in turn yields the coercivity
of $P(\cdot)$ as in \cite{kulkarni2015existence,fabiani2021local}.

\begin{algorithm}[t]
  \caption{Proximal-like LF-NE computation}\label{alg:exact_ordinal}
  \DontPrintSemicolon
  \SetArgSty{}
  \SetKwFor{ForAll}{For all}{do}{End forall}\SetKwFor{While}{While}{do}{End while}
  Set $k = 0$, choose
  $\bs{z}(k)\in \mc Z$, $\tau(k) > 0$ and $\omega \in (0,1)$\\
  \While{$\bs{z}(k)$ is not satisfying a stopping criterion}{
    \ForAll {$i \in [N]$}{
      Obtain $\hat{\bs{z}}^{i}(k)$ and
      set  $z^i(k+1) \in \beta^i_{\tau(k)}(\hat{\bs{z}}^i(k))$\\
      Update $\hat{\bs{z}}^{i+1}(k) = (z^i(k+1), \hat{\bs{z}}^{-i}(k))$
    }
    Set $\bs{z}(k+1)= (z^1(k+1),\ldots,z^N(k+1))$\\
    Update:
    $$         \hspace{-0.15cm}\tau(k +1)\!=\!\textrm{max}\left\{\omega\tau(k),\textrm{min}\{\tau(k),d_{\rho}(\bs{z}(k \!+\!1),\bs{z}(k))\}\right\}$$
    Set $k = k+1$
  }
\end{algorithm}

To compute an equilibrium, we adapt the proximal-type Gauss--Seidel
method defined in \cite{Fabiani_et_al:2022} and reported in
Algorithm~\ref{alg:exact_ordinal}. There, we let $k\geq 0$ denote the
iteration index and $\bs{z}(k)$ the iterate at the beginning of round
$k+1$. For leader $i\in[N]$, the (local) population state reads
\begin{align*}
  &\hat{\bs{z}}^i(k) \\
  &= (z^1(k+1), \ldots,
    z^{i-1}(k+1), z^i(k), z^{i+1}(k),\ldots, z^N(k)).
\end{align*}
This corresponds to the collective vector of strategies in the $k$-th
iteration communicated to the agent $i$ when this agent has to perform
an update. Throughout this process, note that
$\hat{\bs{z}}^{1}(k)=\bs{z}(k)$ and
$\hat{\bs{z}}^{N+1}(k)=\bs{z}(k+1)$.

A key element is the inclusion of an \gls{ICRF} $\rho^i$
\cite[Def.~3]{Fabiani_et_al:2022} that gives (for some $\tau \ge 0$)
rise to the augmented cost function
\begin{equation*}
  \tilde{J}^{i}_{\tau}(z^{i},w^{i};\bs z^{-i})\define
  J^{i}(z^{i}, \bs z^{-i})+\tau \rho^{i}(z^{i}-w^{i})\quad\forall i\in[N],
\end{equation*}
and the corresponding proximal-like \gls{BR} mapping:
\begin{equation*}
  \beta^i_{\tau}(\bs{z}) \define \underset{\zeta^i \in
    \mc{Z}^i}{\textrm{argmin}} \
  \tilde{J}^{i}_{\tau}(\zeta^{i},z^{i};\bs z^{-i}) .
\end{equation*}
Thus, each leader $i \in [N]$ computes the new strategy as a point in the \gls{MI} proximal \gls{BR} mapping,
$
z^i(k+1)\in \beta^i_{\tau(k)}(\hat{\bs{z}}_{i}(k)).
$
Successively, the next internal state $\hat{\bs{z}}^{i+1}(k)$ is updated and passed to the ($i+1$)-th
agent.


The proposed algorithm leverages the adaptive update of the regularization parameter $\tau$, originating a monotonically decreasing sequence $\{\tau(k)\}_{k\geq 0}$, i.e., $\tau(k+1)\leq \tau(k)$ for all $k \geq 0$ \cite{facchinei2011decomposition}.
The rate of decrease of the overall scheme, instead, depends on the "cost-to-move" function $d_{\rho}(\bs{z}(k+1),\bs{z}(k))=\textrm{max}_{i\in[N]} \ \rho^{i}(z^i(k+1)-z^i(k)) \geq 0$, 
which measures the progress the method is making in the agents' proximal steps at the $k$-th iteration, and shall decrease over time to guarantee convergence towards an MI-NE.

\begin{theorem}[\hspace{-.01cm}\textup{\cite[Th.~1]{Fabiani_et_al:2022}}]\label{th:exact_generalized}
Let Assumptions \ref{ass:1}--\ref{ass:compactlevels} hold true. Then, any accumulation point of the sequence $\{\bs{z}(k)\}_{k\geq 0}$ generated by Algorithm~\ref{alg:exact_ordinal} is an MI-NE of the game in \eqref{eq:agent-i-final-reform}, coinciding with an LF-NE of \eqref{eq:Main}.
	\hfill$\square$
\end{theorem}

\begin{remark}
An accumulation point for
$\{\bs{z}(k)\}_{k\geq 0}$ generated by Algorithm \ref{alg:exact_ordinal} exists in view of
Assumption~\ref{ass:compactlevels}.
The analysis performed in \cite{Fabiani_et_al:2022} shows that the sequence $\bs{z}(k)$ monotonically decreases the
potential function values, i.e., $P(\bs{z}(k+1))\leq P(\bs{z}(k))$
holds for all $k\in\N$. Hence, the trajectory $\{\bs{z}(k)\}_{k\in\N}$
belongs to the level set $\mathcal{L}_{P}(\bs{z}(0))$, which we assume
to be a compact set. Existence of converging subsequences is finally
guaranteed by standard arguments involving the Bolzano-Weierstrass theorem.\hfill{$\square$}
\end{remark}

\section{Competition among ride-hailing platforms}
\label{sec:case-studies}




Taking inspiration from
\cite{fabiani2022stochastic,fabiani2023personalized}, we now
investigate how on-demand competing ride-hailing platforms (taking the
role of the leaders) design their pricing strategies under a regulated
pricing scenario. At the same time, a set of subscribed drivers (i.e.,
followers of each platform) decide whether to take a ride or not,
according to the wages they receive and the distance from the pick-up
point.

\subsection{Model description}

\subsubsection{Platforms' game}

Given an area of interest~\mbox{$h\in[H]\subset \N$} (e.g., an airport,
downtown, shopping malls, and so on) we denote with $C^h>0$ the mass
of potential riders in that area. Then, each firm $i \in [N]$ aims at
maximizing its profit by setting a price $p^{i,h} \geq 0$ for the
on-demand ride-hailing service to attract as many customers as
possible in area~$h$. The prices should satisfy a bound of the
form $p^{i,h}\leq\bar p^h$, usually imposed by consumer associations.
The firms must also set a wage $w^{i,h} \in [\underline{w}, \,
\bar{w}^{i,h}]$, with $\underline{w},\bar{w}^{i,h} > 0$ for the
registered drivers on the $i$-th ride-hailing platform. The lower
bound $\underline{w}$ is typically regulated by institutions
\cite{beer2017qualitative,zhong2022demand} and therefore can be taken
uniform across areas.

Similar to \cite{fabiani2022stochastic}, we assume that the fraction
of customers who choose the $i$-th platform’s service in the $h$-th
area, i.e., the demand for the $i$-th firm, is characterized as
\begin{equation*}
  d^{i,h} = \frac{C^h K^{i,h}}{ \bar{p} \sum_{j \in [N]} K^{j,h}}
  \left(\bar{p} - \bar p^{h} + \frac{\theta^i}{N-1} \sum_{j \in
      [N]\setminus\{i\}} p^{j,h} \right),
\end{equation*}
where $K^{i,h} > 0$ denotes the fraction of registered drivers on the
$i$-th ride-hailing platform, based on their willingness to work in
area~$h$ (properly defined in the next section), $\bar{p} >
\textrm{max}_{h \in [H]} \ \bar{p}^h$ is a maximum service price, and
$\theta^i \in [0,1]$ models the substitutability of the service
provided by each firm. The parameter $\theta^i$ being close to~$0$
means that the service of the $i$-th platform is almost independent
from the others, while being close to $1$ means that it is fully
substitutable.

The cost function of each platform can then be written as
\begin{equation}\label{eq:leader_cs}
  F^{i}(x^i,\bs x^{-i},\bs y^i)
  = \sum_{h \in [H]} (w^{i,h} K^{i,h}-p^{i,h} d^{i,h}),
\end{equation}
where $x^i=(p^i,w^i)$ with $p^i=(p^{i,h})_{h\in[H]}$ and $w^i=(w^{i,h})_{h\in[H]}$, and the relation between the followers' decision variables $\bs y^i$ and $K^{i,h}$ is explained in the next section.
The first part of the cost function amounts to the profit of the $i$-th firm to provide a service to the customers, while the second one considers the costs for providing a service to the drivers.


\subsubsection{Drivers' game}
Each platform $i\in[N]$ has $M_i$ subscribed drivers, which have the possibility to choose whether to take a ride or not. In fact, driver $j$ measures the willingness to take a certain ride $y^i_{\nu,h}$,
where $h\in[H]$ is the area of interest.
It follows that the willingness of all the drivers to work in area $h$ for platform $i$ is given by
$$K^{i,h}=\sum_{\nu\in[M_i]}y^i_{\nu,h}.$$

To select a given area, the drivers take into account
a certain preference, according to the function
\begin{equation}\label{eq_cost_pref}
  R^{\nu,h}(y^i_{\nu,h})=B^{\nu,h}y^i_{\nu,h}-A^{\nu,h}(y^i_{\nu,h})^2,
\end{equation}
where $A^{\nu,h}\in\R$ is the willingness to take a ride based on the available rides in a given distance $s^{\nu,h}$:
$$A^{\nu,h}=\frac{C^h}{s^{\nu,h}\beta}.$$
The function in \eqref{eq_cost_pref} expresses the willingness of the drivers to take a ride, namely a driver is ``forced'' to provide service if it is close to the pickup point.
The drivers then aim at maximizing, for a given area $h\in[H]$, the objective function
$$U^{\nu,h}(y^i_{\nu,h})=w^{i,h}y^i_{\nu,h}+R^{\nu,h}(y^i_{\nu,h})-Q^h\sum_{\nu\in[M_i]}y^i_{\nu,h},$$
where $Q^h\sum_{\nu\in[M_i]}y^i_{\nu,h}$ is a congestion-like term modeling the fact that other drivers might want to provide service in the same area $h$, scaled by $Q^h>0$.

Moreover, the drivers subscribed to platform $i$ collectively aim at
satisfying a minimum number of rides, e.g., imposed by the platform
itself, for a given area $h\in[H]$:
\begin{equation}\label{eq_constr_drivers}
  \sum_{\nu\in[M_i]}y^i_{\nu,h}\geq \bar d^{i,h}.
\end{equation}

By defining $y^i_\nu=(y^i_{\nu,h})_{h\in[H]}$, the optimization
problem for driver $\nu\in[M_i]$ subscribed to platform $i$ then reads
\begin{align*}
  \max_{(y^i_{\nu,h})_{h\in[H]}} \quad
  & f^i_{\nu}(x^i,y^i_\nu,\bs y^i_{-\nu})=\sum_{h\in[H]}
    U^{\nu,h}(y^i_{\nu,h})
  \\
  \st \quad
  & \eqref{eq_constr_drivers} \text{ for all } h \in [H].
\end{align*}

\subsection{Numerical results}

\begin{table}
  \caption{Simulation parameters}
  \label{tab:sim_val}
  \centering
  \begin{tabular}{llll}
    \toprule
    Parameter  & Unit & Description   & Value \\
    \midrule
    $\bar{p}$ & \$ & Maximum service price & $32$\\
    $\bar{p}_h$ & \$ & Area price cap & $\sim\mc{U}(16, 30)$\\
    $\underline{w}$ & \$ & Wage lower bound & $12$\\
    $\bar w^{i,h}$ & \$ & Wage upper bound & $\sim\mc{U}(20, 28)$\\
    $\bar \theta$ &  & Substitutability parameter & $0.9$\\
    $C_h$ &  & Area costumers' demand & $\sim\mc{U}(50, 100)$\\
    $B^{\nu,h}$ &  & Linear coefficient ($R^{\nu,h}$) & $\sim\mc{U}(6, 15)$\\
    $s^{\nu,h}$ & [m] & Driver-area distance & $\sim\mc{U}(50, 5000)$\\
    $\beta$ & & Term in $A^{\nu,h}$ & $\sim\mc{U}(0.1, 0.001)$\\
    $Q^h$ & & Congestion coefficient & $\sim\mc{U}(70, 150)$\\
    $\bar d^{i,h}$ & & Minimum number of rides & $20\% C_h$\\
    $\mc M^i_\nu$ & & Upper bound (dual variables) & $200$\\
    $\omega$ & & Scaling factor of Alg.~\ref{alg:exact_ordinal} & $0.1$\\
    \bottomrule
  \end{tabular}
\end{table}
All simulations are run in MATLAB using Gurobi \cite{gurobi} as a
solver on a laptop with an Apple M2 chip featuring an 8-core CPU and
16 GB RAM.
In our tests we have considered a problem instance with $N=5$
ride-hailing platforms, with $M_i=10$ subscribed drivers, who strive
to maximize their profit by operating in $H=3$ different areas. The
remaining numerical values are reported in
Table~\ref{tab:sim_val}. Note that, to make the game an exact
potential one, we set $\theta^i=\bar\theta$ for all
$i\in\{1,\ldots,5\}$. The potential function can then be obtained by
summing \eqref{eq:leader_cs} over all platforms, and considering just
one symmetric term at a time \cite[\S VI-A]{fabiani2023personalized}.
\begin{figure}
  \centering
  \includegraphics[width=\columnwidth]{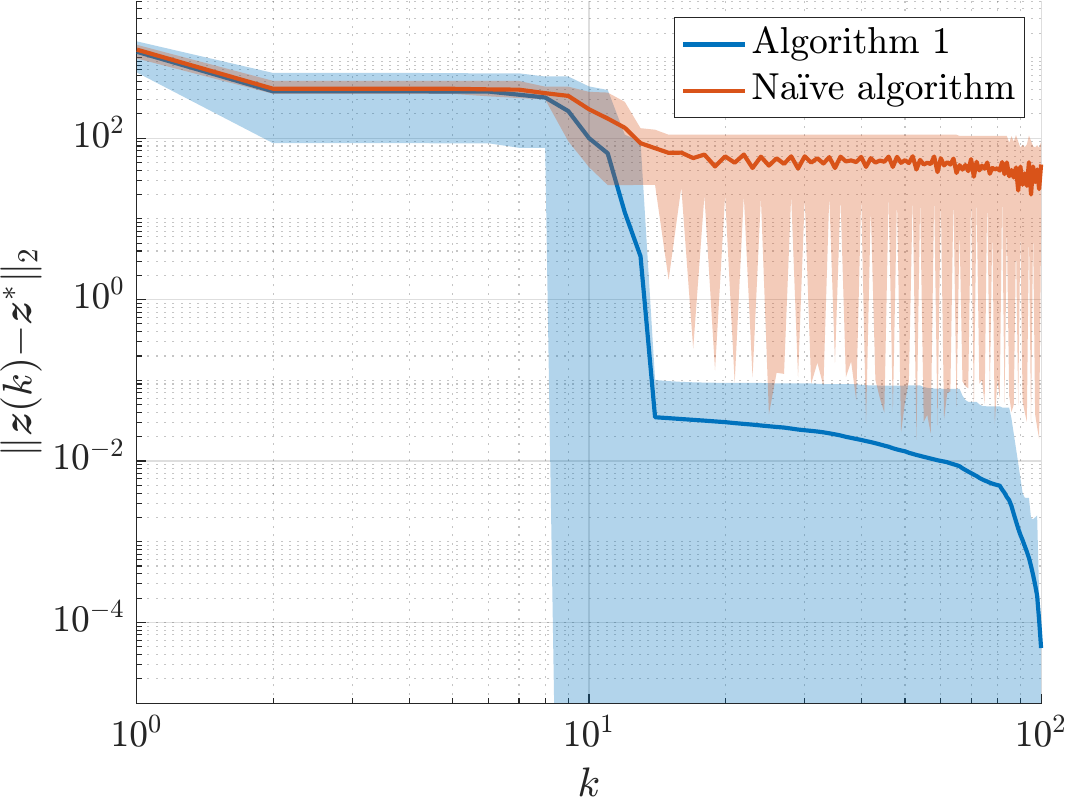}
  \caption{Performance comparison, averaged over $50$ random numerical instances.}
  \label{fig:convergence}
\end{figure}

In Fig.~\ref{fig:convergence} we compare the sequences generated by
Alg.~\ref{alg:exact_ordinal}, and a two-layer procedure
alternating the computation of a Nash equilibrium in the leaders'
game, with fixed followers' strategies, which is then fed into the
(generalized) Nash game involving the followers. In this na\"ive
algorithm we have adopted an iterative, extra-gradient type method
\cite{solodov1996modified} for computing the equilibria in both levels
(i.e., leaders' and followers' game, respectively). Specifically, we
have generated $50$ random instances of the considered game and tested
the behavior of the two procedures. In Fig.~\ref{fig:convergence} one can see that
Alg.~\ref{alg:exact_ordinal} is actually able to return an
equilibrium solution to the multi-leader-multi-follower game, in some
cases with less than $10$ iterations, where each agent takes
$0.0135$\,s to compute a \gls{BR} strategy (worst-case average over
the $50$ numerical instances). Such an equilibrium $\bs z^\ast$ has
been computed by using $\|\cdot\|_2$ as \gls{ICRF}, and its validity
has been verified ex-post by performing an ordered round of BRs. On
the other hand, the na\"ive algorithm exhibits oscillations that
increase in amplitude with~$k$, thereby keeping $\bs z(k)$ far from an
LF-NE.

\begin{table}
  \caption{Numerical results}
  \label{tab:num_results}
  \centering
  \begin{tabular}{p{12mm}ccccc}
    \toprule
    &Firm $1$  & Firm $2$ & Firm $3$ & Firm $4$ & Firm $5$ \\
    \midrule
    $F^{i,\ast}$ (\$) & $307.5$ & $241.2$ & $245.2$ & $207.4$ & $286.1$\\
    \midrule
    Demand \newline satisf. (\%) & $\begin{bmatrix}
      27.3\\
      38.0\\
      29.4\\
    \end{bmatrix}$ & $\begin{bmatrix}
      26.5\\
      38.5\\
      27.5\\
    \end{bmatrix}$ & $\begin{bmatrix}
      26.3\\
      38.5\\
      29.4\\
    \end{bmatrix}$ & $\begin{bmatrix}
      28.2\\
      33.6\\
      27.5\\
    \end{bmatrix}$ & $\begin{bmatrix}
      27.4\\
      34.7\\
      27.4\\
    \end{bmatrix}$ \\
    \bottomrule
  \end{tabular}
\end{table}

The first row in Table~\ref{tab:num_results} reports the
average profit (over the considered $50$ random instances) obtained by
the five companies at the computed LF-NE, where
$F^{i,\ast}=-F^{i}(\bar x^{i,\ast},\bar{\bs x}^{-i,\ast},\bar{\bs y}^{i,\ast})$. This is
made possible since the leaders set competitive wages so that each set
of drivers (i.e., followers) collectively allows to satisfy a certain
amount of costumers' demand per area of interest, as reported in the
second row of Table~\ref{tab:num_results}. Specifically, each quantity
there averages the value $100\times\sum_{\nu\in[M_i]}y^i_{\nu,h}/C_h$
over the $50$ numerical instances, thereby showing that the firms are
actually able to meet more than the $20\,\%$ of $C_h$ in each area, as
imposed through \eqref{eq_constr_drivers}; see the values for $\bar
d^{i,h}$ in Table~\ref{tab:sim_val}.


\section{Conclusion}
\label{sec:conclusion}

We considered a specific class of hierarchical GNEPs in which a
number of leaders play a game while being subject to equilibrium
constraints arising from the decision process of their followers.
For this multi-leader-multi-follower model, we have proposed an
equivalent reformulation that allows the leaders to solve the problem
by estimating the behavior of their followers, and exploiting the
relation between the GNEP and the KKT conditions.
To address the mixed-integer nature of the obtained reformulation and
to find a solution, we restrict the analysis to the class of potential
games, for which iterative schemes are available.
We have verified the convergence of our scheme on a numerical instance
considering a competitive ride-hail market problem.


\bibliographystyle{IEEEtran}
\bibliography{mlmf-via-gs-milp}

\begin{thebibliography}{10}
\providecommand{\url}[1]{#1}
\csname url@samestyle\endcsname
\providecommand{\newblock}{\relax}
\providecommand{\bibinfo}[2]{#2}
\providecommand{\BIBentrySTDinterwordspacing}{\spaceskip=0pt\relax}
\providecommand{\BIBentryALTinterwordstretchfactor}{4}
\providecommand{\BIBentryALTinterwordspacing}{\spaceskip=\fontdimen2\font plus
\BIBentryALTinterwordstretchfactor\fontdimen3\font minus
  \fontdimen4\font\relax}
\providecommand{\BIBforeignlanguage}[2]{{%
\expandafter\ifx\csname l@#1\endcsname\relax
\typeout{** WARNING: IEEEtran.bst: No hyphenation pattern has been}%
\typeout{** loaded for the language `#1'. Using the pattern for}%
\typeout{** the default language instead.}%
\else
\language=\csname l@#1\endcsname
\fi
#2}}
\providecommand{\BIBdecl}{\relax}
\BIBdecl

\bibitem{sherali1984multiple}
H.~D. Sherali, ``A multiple leader {Stackelberg} model and analysis,''
  \emph{Operations Research}, vol.~32, no.~2, pp. 390--404, 1984.

\bibitem{hintermuller2014several}
M.~Hinterm{\"u}ller, B.~S. Mordukhovich, and T.~M. Surowiec, ``Several
  approaches for the derivation of stationarity conditions for elliptic {MPECs}
  with upper-level control constraints,'' \emph{Mathematical Programming}, vol.
  146, no. 1-2, pp. 555--582, 2014.

\bibitem{cui2023computation}
S.~Cui and U.~V. Shanbhag, ``On the computation of equilibria in monotone and
  potential stochastic hierarchical games,'' \emph{Mathematical Programming},
  vol. 198, no.~2, pp. 1227--1285, 2023.

\bibitem{cui2023regularized}
S.~Cui, U.~V. Shanbhag, and M.~Staudigl, ``A regularized variance-reduced
  modified extragradient method for stochastic hierarchical games,''
  \emph{arXiv preprint arXiv:2302.06497}, 2023.

\bibitem{luo1996mathematical}
Z.-Q. Luo, J.-S. Pang, and D.~Ralph, \emph{Mathematical programs with
  equilibrium constraints}.\hskip 1em plus 0.5em minus 0.4em\relax Cambridge
  University Press, 1996.

\bibitem{hintermuller2011smooth}
M.~Hinterm{\"u}ller and I.~Kopacka, ``A smooth penalty approach and a nonlinear
  multigrid algorithm for elliptic {MPECs},'' \emph{Computational Optimization
  and Applications}, vol.~50, pp. 111--145, 2011.

\bibitem{de2023bilevel}
J.~C. De~los Reyes, ``Bilevel imaging learning problems as mathematical
  programs with complementarity constraints: Reformulation and theory,''
  \emph{SIAM Journal on Imaging Sciences}, vol.~16, no.~3, pp. 1655--1686,
  2023.

\bibitem{Aussel:2020aa}
D.~Aussel and A.~Svensson, ``A short state of the art on multi-leader-follower
  games,'' \emph{Bilevel optimization: Advances and next challenges}, pp.
  53--76, 2020.

\bibitem{Pang:2005aa}
J.-S. Pang and M.~Fukushima, ``Quasi-variational inequalities, generalized
  {Nash} equilibria, and multi-leader-follower games,'' \emph{Computational
  Management Science}, vol.~2, no.~1, pp. 21--56, 2005.

\bibitem{fabiani2021local}
F.~Fabiani, M.~A. Tajeddini, H.~Kebriaei, and S.~Grammatico, ``Local
  {Stackelberg} equilibrium seeking in generalized aggregative games,''
  \emph{IEEE Transactions on Automatic Control}, vol.~67, no.~2, pp. 965--970,
  2021.

\bibitem{kulkarni2015existence}
A.~A. Kulkarni and U.~V. Shanbhag, ``An existence result for hierarchical
  {S}tackelberg vs {S}tackelberg games,'' \emph{IEEE Transactions on Automatic
  Control}, vol.~60, no.~12, pp. 3379--3384, 2015.

\bibitem{BlumeZame94}
L.~E. Blume and W.~R. Zame, ``The algebraic geometry of perfect and sequential
  equilibrium,'' \emph{Econometrica}, vol.~62, no.~4, pp. 783--794, 1994.

\bibitem{Fabiani_et_al:2022}
F.~Fabiani, B.~Franci, S.~Sagratella, M.~Schmidt, and M.~Staudigl,
  ``Proximal-like algorithms for equilibrium seeking in mixed-integer {N}ash
  equilibrium problems,'' in \emph{2022 IEEE 61st Conference on Decision and
  Control (CDC)}, 2022, pp. 4137--4142.

\bibitem{sagratella2017algorithms}
S.~Sagratella, ``Algorithms for generalized potential games with mixed-integer
  variables,'' \emph{Computational Optimization and Applications}, vol.~68,
  no.~3, pp. 689--717, 2017.

\bibitem{sagratella2019generalized}
------, ``On generalized {Nash} equilibrium problems with linear coupling
  constraints and mixed-integer variables,'' \emph{Optimization}, vol.~68,
  no.~1, pp. 197--226, 2019.

\bibitem{fabiani2022stochastic}
F.~Fabiani and B.~Franci, ``A stochastic generalized {N}ash equilibrium model
  for platforms competition in the ride-hail market,'' in \emph{2022 IEEE 61st
  Conference on Decision and Control (CDC)}.\hskip 1em plus 0.5em minus
  0.4em\relax IEEE, 2022, pp. 4455--4460.

\bibitem{fabiani2023personalized}
F.~Fabiani, A.~Simonetto, and P.~J. Goulart, ``Personalized incentives as
  feedback design in generalized {Nash} equilibrium problems,'' \emph{IEEE
  Transactions on Automatic Control}, vol.~68, no.~12, pp. 7724--7739, 2023.

\bibitem{Aussel_et_al:2023}
\BIBentryALTinterwordspacing
D.~Aussel, C.~Egea, and M.~Schmidt, ``A tutorial on solving
  single-leader-multi-follower problems using {SOS1} reformulations,'' Tech.
  Rep., 2023. [Online]. Available:
  \url{https://optimization-online.org/?p=23744}
\BIBentrySTDinterwordspacing

\bibitem{Facchinei_2010_GNEP}
F.~Facchinei and C.~Kanzow, ``Generalized {Nash} equilibrium problems,''
  \emph{Annals of Operations Research}, vol. 175, no.~1, pp. 177--211, 2010.

\bibitem{facchinei2011decomposition}
F.~Facchinei, V.~Piccialli, and M.~Sciandrone, ``Decomposition algorithms for
  generalized potential games,'' \emph{Computational Optimization and
  Applications}, vol.~50, no.~2, pp. 237--262, 2011.

\bibitem{Dempe-et-al:2015}
S.~Dempe, V.~Kalashnikov, G.~A. P{\'e}rez-Vald{\'e}s, and N.~Kalashnykova,
  \emph{Bilevel Programming Problems}.\hskip 1em plus 0.5em minus 0.4em\relax
  Springer, 2015.

\bibitem{Dempe_Zemkoho:2020}
S.~Dempe and A.~Zemkoho, Eds., \emph{Bilevel Optimization: Advances and Next
  Challenges}.\hskip 1em plus 0.5em minus 0.4em\relax Springer International
  Publishing, 2020.

\bibitem{Aussel-et-al:2019}
D.~Aussel and A.~Svensson, ``Towards tractable constraint qualifications for
  parametric optimisation problems and applications to generalised {Nash}
  games,'' \emph{Journal of Optimization Theory and Applications}, vol. 182,
  pp. 404--416, 2019.

\bibitem{Dempe-Dutta:2012}
S.~Dempe and J.~Dutta, ``Is bilevel programming a special case of a
  mathematical program with complementarity constraints?'' \emph{Mathematical
  Programming}, vol. 131, no. 1-2, pp. 37--48, 2012.

\bibitem{beer2017qualitative}
R.~Beer, C.~Brakewood, S.~Rahman, and J.~Viscardi, ``Qualitative analysis of
  ride-hailing regulations in major {American} cities,'' \emph{Transportation
  Research Record}, vol. 2650, no.~1, pp. 84--91, 2017.

\bibitem{zhong2022demand}
Y.~Zhong, T.~Yang, B.~Cao, and T.~Cheng, ``On-demand ride-hailing platforms in
  competition with the taxi industry: {P}ricing strategies and government
  supervision,'' \emph{International Journal of Production Economics}, vol.
  243, p. 108301, 2022.

\bibitem{gurobi}
\BIBentryALTinterwordspacing
{Gurobi Optimization, LLC}, ``{Gurobi Optimizer Reference Manual},'' 2023.
  [Online]. Available: \url{https://www.gurobi.com}
\BIBentrySTDinterwordspacing

\bibitem{solodov1996modified}
M.~V. Solodov and P.~Tseng, ``Modified projection-type methods for monotone
  variational inequalities,'' \emph{SIAM Journal on Control and Optimization},
  vol.~34, no.~5, pp. 1814--1830, 1996.

\end{thebibliography}
\end{document}